\documentclass[11pt]{amsart}

\usepackage{amsmath,amssymb}
\usepackage{indentfirst}
\usepackage[all]{xy}
\usepackage{color}
\newtheorem{theorem}{Theorem}[section]
\newtheorem{lemma}[theorem]{Lemma}
\newtheorem{corollary}[theorem]{Corollary}

\newtheorem{definition}[theorem]{Definition}

\newtheorem{proposition}[theorem]{Proposition}

\newcommand{\CW}{\operatorname{CW}}

\newcommand{\N}{\mathbb{N}}

\newcommand{\R}{\operatorname{{\mathbb R}}}
\newcommand{\Z}{\operatorname{{\mathbb Z}}}
\renewcommand{\N}{\operatorname{{\mathbb N}}}
\newcommand{\Q}{\operatorname{{\mathbb Q}}}

\newcommand\ANR{\operatorname{ANR}}

\newcommand\inter{\operatorname{int}}

\newcommand\extdim{\operatorname{extdim}}

\newcommand\mesh{\operatorname{mesh}}
\newcommand\diam{\operatorname{diam}}

\newcommand\st{\operatorname{st}}

\newcommand\inv{^{-1}}

\begin{document}

\title{Simplicial Inverse Sequences in Extension Theory}
\author{Leonard R. Rubin}

\address{Department of Mathematics\\
University of Oklahoma\\
Norman, Oklahoma 73019\\
USA}\email{lrubin@ou.edu}

\author{Vera Toni\' c}

\address{Department of Mathematics\\
University of Rijeka\\
51000 Rijeka\\ Croatia}\email{vera.tonic@math.uniri.hr}

\date {13 March 2017}

\begin{abstract}In extension theory, in particular in
dimension theory, it is frequently useful
to represent a given compact metrizable space $X$ as the
limit of an inverse sequence of compact polyhedra.  We
are going to show that, for the purposes of extension theory,
it is possible to replace such an $X$ by a better metrizable
compactum $Z$.  This $Z$ will come as the limit of an inverse
sequence of triangulated polyhedra with  simplicial
bonding maps that factor in a certain way.  There will be a
cell-like map $\pi:Z\to X$, and we shall show that if $K$
is a $\CW$-complex with $X\tau K$, then $Z\tau K$.
\end{abstract}



\maketitle \markboth{L. Rubin and V. Toni\' c}{Simplicial
Inverse Sequences in Extension Theory}

\section[Introduction]
{Introduction}\label{intro}

In extension theory, and in particular the theories of
dimension and cohomological dimension $\dim_G$ over an abelian
group $G$ (\cite{Ku}), it is frequently useful to represent a
given compact metrizable space $X$ as the limit of an inverse
sequence of compact polyhedra. This was of importance in the
proofs of the Edwards-Walsh cell-like resolution theorem
(\cite{Ed}, \cite{Wa}), Dranishnikov's
$\mathbb{Z}/p$-resolution theorem (\cite{Dr}), and Levin's
$\mathbb{Q}$-resolution theorem (\cite{Le}).  In each case the
hypothesis was that $\dim_G X\leq n$ (the abelian group $G$
depending on which of the three cases was under consideration),
and the first step in their respective proofs of the existence
of a resolution of the desired type (cell-like, $\Z/p$-acyclic,
$\Q$-acyclic, respectively) was to represent $X$ as the
inverse limit of an inverse sequence of compact polyhedra. That
this can always be done comes originally from H. Freudenthal
(\cite{Fr}), but the result can be found also as Corollary 4.10.11.
in \cite{Sa}. It stipulates that each compact metrizable space
can be written as the inverse limit of an inverse sequence
$(X_i,p_i^{i+1})$ of finite polyhedra with surjective piecewise
linear bonding maps $p_i^{i+1}:X_{i+1}\to X_i$, where piecewise
linear means that the domain and range of $p_i^{i+1}$ can be
triangulated in such a manner that $p_i^{i+1}$ is simplicial
with respect to these triangulations.

One might ask if it is possible to arrange such an inverse
sequence so that each polyhedron has a fixed triangulation and
so that all the bonding maps are simplicial with respect to
these triangulations. It was shown in \cite{Ma} that this is
not always attainable.  On the other hand, at least for the
purposes of extension theory, can such an obstacle be removed?
Let us give a brief explanation of what we have accomplished
in this direction.

When $X$ and $K$ are spaces, we are going to write $X\tau K$ to
mean that $X$ is an absolute co-extensor for $K$, i.e., for
each closed subset $A$ of $X$ and map $f:A\to K$, there exists
a map $g:X\to K$ that extends $f$.  This is the fundamental
notion of extension theory, and typically $K$ is a
$\CW$-complex.  Let $X$ be a compact metrizable space.  Then
$X\tau S^n$ if and only if $\dim X\leq n$.  For cohomological
dimension $\dim_G$ over an abelian group $G$, one has that
$\dim_G X\leq n$ if and only if $X\tau K(G,n)$ where $K(G,n)$
is any Eilenberg-MacLane complex of type $(G,n)$. Thus
extension theory is a unifying structure in the study of
dimension theory.

Our main result is Theorem \ref{bigmapistrivsh}.  It states
that for a given nonempty metrizable compactum $X$, there
exists a metrizable compactum $Z$ and a cell-like map (see
Definition \ref{trivshamap}) $\pi:Z\to X$ such that if $K$ is a
$\CW$-complex with $X\tau K$, then $Z\tau K$. Moreover, the
compactum $Z$ comes as the limit of an inverse sequence
$\mathbf{Z}=(|T_j|,g_j^{j+1})$ in which all the bonding maps
 are simplicial with respect to the given
finite triangulations $T_i$ of the polyhedra $|T_i|$.
These $g_j^{j+1}$ have simplicial factorizations (see
Lemma \ref{enhancedbigpicture}) $g_j^{j+1}=f_j^{j+1}\circ
\varphi_{j+1}:|T_{j+1}|\to|T_j|$, $\varphi_{j+1}:|T_{j+1}|
\to|\widetilde T_{j+1}|$, where $T_{j+1}$ is a subdivision
of $\widetilde T_{j+1}$ and $\varphi_{j+1}$ is a simplicial
approximation to the identity map; these play a prominent
role in our development.

We also provide a theory of ``adjustments'' (Definition
\ref{adjustedinvseq}) to such a $\mathbf{Z}$. If $n\geq0$,
$(j_i)$ is an increasing sequence in $\N$, and each
$g_{j_i}^{j_{i+1}}\big||T_{j_{i+1}}^{(n)}|:
|T_{j_{i+1}}^{(n)}|\to|T_{j_i}^{(n)}|$, is replaced by a map
$h_i^{i+1}$ that is a $T_{j_i}$-modification of it (see
Definition \ref{modif}), then we get a new inverse sequence
$\mathbf{M}=(|T_{j_i}^{(n)}|,h_i^{i+1})$. There is a uniquely
induced surjective map $\pi:\lim\mathbf{M}\to X$. This is
covered in Lemma \ref{altAJR}, where it is shown how to
describe each fiber $\pi\inv(x)$ of $\pi$ as the limit of three
different sub-inverse sequences of $\mathbf{M}$. This was
employed in our proof of Theorem \ref{bigmapistrivsh}. The
existence of such maps whose fibers are so well-described
has the potential to be used in other resolution theorems.

\section[Simplicial Complexes and Extensors]
{Simplicial Complexes and Extensors}\label{simpcomp}

For each simplicial complex $T$,  $|T|$ will designate its
polyhedron with the weak topology.  The $n$-skeleton of $T$ is
going to be written $T^{(n)}$.  If $T$ is finite, then we shall
supply it with the metric induced by $T$; in this case the weak
topology on $|T|$ is the same as the metric topology. If $v\in
T^{(0)}$, then $\st(v,T)$ will denote the {\it open star} of
$v$ in $T$ and $\overline\st(v,T)$ will denote the {\it closed
star} of $v$ in $T$.  Of course, $\st(v,T)$ is an open
neighborhood of $v$ in the polyhedron $|T|$,
$\overline\st(v,T)$ is a closed subset of $|T|$, and
$\st(v,T)\subset\overline\st(v,T)$. Moreover, each of
$\st(v,T)$ and $\overline\st(v,T)$ is contractible, and there
is a unique subcomplex $S_{v,T}$ of $T$ such that $|S_{v,T}|=
\overline\st(v,T)$.  We make the convention that
$T^{(\infty)}=T$.  Map will always mean continuous function.
Also,

\begin{definition}\label{modif} If $f:X\to|T|$ is a function
where $T$ is a simplicial complex, then a function $g:X\to|T|$
is called a $T$-{\bf modification} of $f$ if for each $x\in X$
and simplex $\sigma$ of $T$ with $f(x)\in\sigma$,
$g(x)\in\sigma$. This is equivalent to saying that for each
$x\in X$ and simplex $\sigma$ of $T$ with
$f(x)\in\inter(\sigma)$, $g(x)\in\sigma$.
\end{definition}

The ``straight
line'' homotopy gives us the next fact.

\begin{lemma}\label{modishotop}Let $T$ be a finite simplicial complex,
$X$ a space, and $f$, $g$ maps of $X$ to $|T|$ such that $g$
is a $T$-modification of $f$. Then $g\simeq f$.\qed
\end{lemma}

\begin{definition}\label{simpnbhd}For each simplicial complex
$T$ and nonempty subset $D\subset|T|$, we shall denote by
$N_{D,T}$ the subcomplex of $T$ consisting of the simplexes of
$T$ that intersect $D$ and all faces of such simplexes.  This
is the {\bf simplicial neighborhood} of $D$ in $T$.
\end{definition}

Let us review the notion of {\it extensor} (\cite{Hu}).  A
space $K$ is an {\it absolute neighborhood extensor} for a
space $X$, written $K\in\mathrm{ANE}(X)$, if each map of a
closed subspace $A$ of $X$ to $K$ extends to a map of a
neighborhood of $A$ in $X$ to $K$. An $\ANR$, absolute
neighborhood retract, is a metrizable
space that is an absolute neighborhood extensor for any
metrizable space.

We state a
version of (R1) from page 74 of \cite{MS} that will be suitable
for our purposes.

\begin{lemma}\label{MarSegRes}Let $\mathbf{Y}=(Y_i,g_i^{i+1})$
be an inverse sequence of metrizable compacta,
$Y=\lim\mathbf{Y}$, $P$ an $\ANR$ with metric $d$, $\mu:X\to P$
a map, and $\epsilon>0$. Then there exists $i$ such that for
all $j\geq i$, there is a map $f:Y_j\to P$ with $d(f\circ
g_{j,\infty},\mu)<\epsilon$.
\end{lemma}

\begin{lemma}\label{closearehomtop}Let $K$ be a $\CW$-complex.
Then $K$ has an open cover $\mathcal{V}$ such that any two
$\mathcal{V}$-close maps of any space to $K$ are homotopic.
\end{lemma}

\begin{proof}There exists a simplicial complex $L$ such that
$|L|_m$, that is $|L|$ with the metric topology, is homotopy
equivalent to $K$.  Choose a homotopy equivalence
$f:K\to|L|_m$. By Theorem III.11.3 (page 106) of \cite{Hu},
$|L|_m$ is an ANR.  Theorem IV.1.1 (page 111) of \cite{Hu}
shows that there is an open cover $\mathcal{W}$ of $|L|_m$
having the property that any two $\mathcal{W}$-close maps of
any space to $|L|_m$ are homotopic.  The open cover needed for
$K$ is $\mathcal{V}=f\inv(\mathcal{W})$.
\end{proof}

There is a relatively standard technique that can be used to help
detect when the limit of an inverse sequence of metrizable compacta
is an absolute co-extensor for a given $\CW$-complex.  Here are the
needed concepts.

\begin{proposition}\label{sequenceextension}Let $\mathbf{Z}=
(Z_i,g_i^{i+1})$ be an inverse sequence of nonempty compact
metrizable spaces, $Z=\lim\mathbf{Z}$, and $K$ a $\CW$-complex.
Suppose that for each $i\in\N$, closed subset $D$ of $Z_i$, and
map $f:D\to K$, there exist $j\geq i$ and a map $g:Z_j\to K$
such that for all $x\in(g_i^j)\inv(D)$, $g(x)=f\circ g_i^j(x)$.
Then $Z\tau K$.\qed
\end{proposition}

\begin{definition}\label{baseofclosed}Let $Z$ be a nonempty space and
$\mathcal{B}$ a collection of nonempty closed subsets of $Z$. We shall
call $\mathcal{B}$ a {\bf base} for the closed subsets of
$Z$ provided that for each closed subset
$C$ of $Z$ and neighborhood $U$ of $C$ in $Z$, there exists
$A\in\mathcal{B}$ with $C\subset\inter_Z A\subset A\subset U$.
\end{definition}

\begin{lemma}\label{countclosed}Every compact metrizable space has
a countable base for its closed subsets.\qed
\end{lemma}

\begin{definition}\label{fixedcountbase}For each compact metrizable
space $X$, let $\mathcal{B}(X)$ designate a fixed countable base
for the closed subsets of $X$.
\end{definition}

\begin{proposition}\label{baseextension}Let $\mathbf{Z}=(Z_i,g_i^{i+1})$ be
an inverse sequence of compact metrizable spaces,
$Z=\lim\mathbf{Z}$, and $K$ a $\CW$-complex.  Suppose that for
each $i\in\N$, $D\in\mathcal{B}(Z_i)$, and map $f:D\to K$,
there exist $j\geq i$ and a map $g:Z_j\to K$ such that for all
$x\in(g_i^j)\inv(D)$, $g(x)=f\circ g_i^j(x)$.  Then $Z\tau
K$.\qed
\end{proposition}

We need to organize certain collections of homotopy classes.  Here is the
fundamental fact.

\begin{lemma}\label{counthomotop}For each compact metrizable space $X$
and compact polyhedron $P$, the set $[X,P]$ of homotopy classes of maps
of $X$ to $P$ is countable.\qed
\end{lemma}

\begin{definition}\label{selectclass}For each compact metrizable space $X$
and compact polyhedron $P$, select a countable collection $\mathcal{H}(X,P)$
consisting of one representative from each homotopy class in $[X,P]$.
\end{definition}

\section[Extensor Lemma]
{Extensor Lemma}\label{ExtenLem}

For the remainder of this paper, $I^\infty$ will denote the
Hilbert cube, i.e., $I^\infty=\prod\{I_i\,|\,i\in\N\}$ where
$I_i=I$ for each $i$.  For each $j\in\N$, we factor $I^\infty$ as
$I^j\times I^\infty_j$.  Let $0_j$
denote the element of $I^\infty_j$ each of whose coordinates is $0$.
If $P\subset I^j$, then we may treat $P$ as
$P\times\{0_j\}\subset I^\infty$.  The context should make this clear
when we apply it.  In this setting, an element of $P$ becomes
the element of $I^\infty$ whose first $j$ coordinates are the ones it
inherits from $P$ and whose remaining coordinates are all $0$.

We shall use the metric $\rho$ on
$I^\infty$ given by
$\rho(x,y)=\sum_{i=1}^{\infty}\frac{\vert x_i-y_i\vert}{2^i}$.
For each $k\in\N$, $p_{k,\infty}:I^\infty\to I^k$ will denote
the $k$-coordinate projection map,
and if $j\leq k$, we will use $p_j^k:I^k\to I^j$ for the
$j$-coordinate projection map. In case $x\in I^k$, then according to
our convention $x=(x,0,0,\dots)\in I^\infty$, so one
has that $p_j^k(x)=p_{j,\infty}(x)\in I^j$.

The main result of this section is Lemma \ref{endset}.  It provides us
with a statement, see (3), about extending a map under the condition
that a given compactum $X$ has been embedded in $I^\infty$.

\begin{lemma}\label{intersectinH} Let $X\subset I^\infty$ be
compact and nonempty.  Then there exist an increasing
sequence $(n_j)$ in $\N$ with $n_1=1$,
and a sequence $(P_j)$ of compact polyhedra $P_j\subset I^{n_j}$,
such that:

\begin{enumerate}
\item for all $j\in\N$, $X\subset\inter_{I^\infty}(P_j\times I^\infty_{n_j})
\subset N(X,\frac{2}{j})$, and\item if $j\in\N_{\geq2}$, then
$p_{n_{j-1}}^{n_j}(P_j)\subset\inter_{I^{n_{j-1}}}P_{j-1}$.
\end{enumerate}
\end{lemma}

\begin{proof}
Put $n_1=1$ and $P_1=I^1$.  Then $(1)$ is true
in case $j=1$ since $\diam I^\infty\leq1$, and $(2)$
is true vacuously. Proceed by induction.  Suppose that
$j\in\N$, and we have found finite sequences $n_1<\dots<n_j$
in $\N$ and compact polyhedra $P_1,\dots, P_j$ such that
for $1\leq s\leq k$, $P_s\subset I^{n_s}$, $(1)$ is true
up to $j$ and $(2)$ is true up to $j-1$ in case $1<j$.

One may choose $n_{j+1}\in\N$ such that $n_{j+1}>n_j$ and $X\subset
p_{n_{j+1},\infty}(X)\times I^\infty_{n_{j+1}}\subset N(X,\frac{2}{j+1})$.
There is a neighborhood $V$ of $p_{n_{j+1},\infty}(X)$ such that
$V\times I^\infty_{n_{j+1}}\subset N(X,\frac{2}{j+1})$. Choose a compact
polyhedron $P_{j+1}$ of $I^{n_{j+1}}$ so that,
$p_{n_{j+1},\infty}(X)\subset\inter_{I^{n_{j+1}}} P_{j+1}
\subset P_{j+1}\subset V$. This gives us $(1)$ for $j+1$.

Notice that $(1)$ for $j$ implies,
$p_{n_j,\infty}(X)=p_{n_j}^{n_{j+1}}\circ p_{n_{j+1},\infty}(X)\subset
\inter_{I^{n_j}}P_j$. Hence,
$p_{n_{j+1},\infty}(X)\subset(p_{n_j}^{n_{j+1}})^{-1}(\inter_{I^{n_j}}(P_j))$.
Thus, making $P_{j+1}$ smaller if necessary, we may have $(1)$ and
simultaneously, $P_{j+1}\subset(p_{n_j}^{n_{j+1}})^{-1}(\inter_{I^{n_j}}(P_j))$.
This achieves $(2)$ for $j+1$.
\end{proof}

The condition (2) of Lemma \ref{intersectinH} shows that if we replace
$j$ by $j+1$, we find that $p_{n_j}^{n_{j+1}}(P_{j+1})\subset
\inter_{I^{n_j}}P_j$.  This implies that $P_{j+1}\times I_{n_{j+1}}^\infty
\subset(\inter_{I^{n_j}}P_j)\times I_{n_j}^\infty$.  This and (1) of
Lemma \ref{intersectinH} lead us to the next piece of information.

\begin{corollary}\label{picture1}In the setting of
$\mathrm{Lemma\,\ref{intersectinH}}$,
\begin{enumerate}\item for each $j\in\N$, $P_{j+1}\times I^\infty
_{n_{j+1}}\subset\inter_{I^\infty}(P_j\times I^\infty_{n_j})$, and
\item $X=\bigcap\{P_j\times I^\infty_{n_j}\,|\,j\in\N\}$.\qed
\end{enumerate}
\end{corollary}

In reading (3) of the ensuing lemma, one should consult (2) of
Lemma \ref{intersectinH} to see that whenever $j\leq l$, then
$p_{n_j}^{n_l}(P_l)\subset P_j$.

\begin{lemma}\label{endset}Let $X\subset I^\infty$,
$(n_j)$, $(P_j)$ be as in
$\mathrm{Lemma\,\,\ref{intersectinH}}$, and $K$ be a $\CW$-complex
with $X\tau K$.  Suppose that $j\in\N$ and $B_j$ is a closed
subset of $P_j$.  For each $k\geq j$, let $B_k=
(p_{n_j}^{n_k})\inv(B_j)\cap P_k$
and put $B_{j,\infty}=p_{n_j,\infty}\inv(B_j)\cap X$.
The following are true.
\begin{enumerate}\item For any open
neighborhood $S$ of $B_{j,\infty}$ in $I^\infty$, there exists $k\geq j$ such
that for all $l\geq k$, $B_l\subset S$.\footnote{Statement (1) is true
independently of the condition $X\tau K$.}
\item If $f:B_j\to K$ is
a map, then there exists $k\geq j$ such that for all $l\geq k$, there
is a map $f^*:P_l\to K$ that extends the composition $f\circ
p_{n_j}^{n_l}|B_l:B_l\to K$ where we treat $p_{n_j}^{n_l}
|B_l:B_l\to B_j$.\item Suppose that $j\in\N$,
$T_j$ is a triangulation of $P_j$,
$L$ is a subcomplex of $T_j$, and that $|N_{|L|,T_j}|$,
$N_{|L|,T_j}$ being the simplicial
neighborhood of $|L|$ in $T_j$,
is a regular neighborhood of $|L|$ in $|T_j|$.  Assume that
$k\geq j$ is as in $(2)$ with $B_j=|L|$, $l\geq k$,
and $g:P_l\to|T_j|$ is a map which is a $T_j$-modification of
$p_{n_j}^{n_l}|P_l:P_l\to P_j=|T_j|$.
Let $f:|L|\to K$ be a map and $E=g\inv(|L|)\subset P_l$. Then there is a map
$g^*:P_l\to K$ that extends the composition $f\circ g|E:E\to K$.
\end{enumerate}
\end{lemma}

\begin{proof}Let $S$ be an
open neighborhood of $B_{j,\infty}$ in
$I^\infty$.  If the conclusion of (1) is not true, then there
is an increasing sequence $(m_i)$ in $\N$, $m_1\geq j$, so that
for each $i$, there exists $b_i\in B_{m_i}\setminus S
\subset P_{m_i}$.  Passing to a subsequence if necessary,
we may assume that the sequence $(b_i)$ in the compactum
$I^\infty\setminus S$ converges in $I^\infty$ to $b\in I^\infty\setminus S$.
Applying Corollary \ref{picture1}(1,2) along with the fact
that $b_i\in P_{m_i}$, one sees that $b\in X
\setminus B_{j,\infty}$,
from which we deduce that $p_{n_j,\infty}(b)\notin B_j$.

For each $i$, $p_{n_j}^{s_i}(b_i)=p_{n_j,\infty}(b_i)\in B_j$,
$s_i=n_{m_i}$.
Hence $\{p_{n_j,\infty}(b_i)\,|\,i\in\N\}\subset B_j$.  Since
$B_j$ is closed in $P_j$, $p_{n_j,\infty}$ is a map, and
$(b_i)$ converges to $b$, then $p_{n_j,\infty}(b)\in B_j$, a
contradiction.  This yields (1).  Now to prove (2).

Employing Lemma \ref{closearehomtop},
select an open cover $\mathcal{V}$ of $K$
such that for any space $Y$, any maps $g_1:Y\to K$ and $g_2:Y\to K$
that are $\mathcal{V}$-close are homotopic.  Let $\mathcal{V}_1$
be an open cover of $K$ that star-refines $\mathcal{V}$.
Choose an open cover $\mathcal{W}$ of $B_j$ such that if
$W\in\mathcal{W}$, then there
exists $V_W \in\mathcal{V}_1$ with $f(W)\subset V_W$.

Observe that $B_{j,\infty}$ is a closed subset of $X$ and
that $p_{n_j,\infty}(B_{j,\infty})\subset B_j$.
Since $X\tau K$, then the map $f\circ p_{n_j,\infty}
|B_{j,\infty}:B_{j,\infty}\to K$
extends to a map $h:U\to K$ where $U$ is an open neighborhood
of $X$ in $I^\infty$.  Select an open cover $\mathcal{R}$ of
$X$ in $U$ having the property that for each
$R\in\mathcal{R}$, there exists $V_R\in\mathcal{V}_1$ with
$h(R)\subset V_R$.  Let $S=\bigcup\mathcal{R}\subset U$.  Then
$S$ is an open neighborhood of $X$ in $I^\infty$.  So
by (1), we may choose $k\in\N$ so that for all $l\geq k$,
$B_l\subset S$.  Using Corollary \ref{picture1}(1,2), we may
also require that for such $l$, $P_l=P_l\times\{0_{n_l}\}\subset S$.

Put $B^*=B_{j,\infty}\cup\bigcup\{B_l\,|\,l\geq k\}\subset S$.  Then
of course $p_{j,\infty}:B^*\to B_j$ is a map.
For each $b\in B_{j,\infty}$,
select a neighborhood $E_b$ of $b$ in $B^*$ such that $p_{j,\infty}(E_b)$
is contained in an element $W_b$ of $\mathcal{W}$ and that in
addition, there exists $R_b\in\mathcal{R}$ with $E_b\subset R_b$.
Let $S_0=\bigcup\{E_b\,|\,b\in B_{j,\infty}\}$.  Then $S_0$ is an
open neighborhood of $B_{j,\infty}$ in $B^*\subset I^\infty$.  So
there is an open subset $S_1$ of $I^\infty$ having the property
that $S_1\cap B^*=S_0$.  Plainly, $S_1$ is an open neighborhood
of $B_{j,\infty}$ in $I^\infty$.
An application of (1) with $S_1$ in place of $S$ gives
us
the existence of a $k_1\geq k$ so that for all $l\geq k_1$,
we have $B_l\subset S_1$.
But then $B_l\subset S_1\cap B^*=S_0$.

We are going to show that
$f\circ p_{n_j}^{n_l}|B_l:B_l\to K$ is homotopic to $h|B_l:B_l\to K$.
For in that case, if we define $h_0=h|B_l:B_l\to K$, then of course
since $P_l\subset S$, $h_0$ extends to the map
$h|P_l:P_l\to K$, and the homotopy
extension theorem will complete our proof.

Let $x\in B_l$.  It will be sufficient to show that
$f\circ p_{n_j}^{n_l}(x)$ and $h_0(x)$
lie in an element of $\mathcal{V}$.
There exists $b\in B_{j,\infty}$ such that $x\in E_b$.  Now
$b\in E_b\subset R_b\in\mathcal{R}$.  It follows that there
is an element $V_1\in\mathcal{V}_1$ with $\{h(b),h(x)\}=
\{h(b),h_0(x)\}\subset V_1$.  One sees from the
definition of $h$ and the fact
that $b\in B_{j,\infty}$, that $h(b)=f\circ p_{j,\infty}(b)$.
So we have that $\{f\circ p_{j,\infty}(b),h_0(x)\}
\subset V_1\in\mathcal{V}_1$.
We know that $p_{j,\infty}(b)\in p_{j,\infty}(E_b)\subset W_b\in\mathcal{W}$.
Thus $f\circ p_{j,\infty}(b)\in f\circ p_{j,\infty}(E_b)
\subset f(W_b)\subset V_2$ for some $V_2\in\mathcal{V}_1$.
Now $x\in E_b\cap B_l\subset E_b\cap P_l\subset
E_b\cap I^{n_l}$, so $p_{j,\infty}(x)=p_{n_j}^{n_l}(x)\in
p_{j,\infty}(E_b)$, and we see that $f\circ p_{n_j}^{n_l}(x)\in
f\circ p_{j,\infty}(E_b)\subset V_2$.
Hence, $\{f\circ p_{j,\infty}(b),f\circ p_{n_j}^{n_l}(x)\}\subset
V_2$.  Since $\mathcal{V}_1$ is a star-refinement of $\mathcal{V}$,
$f\circ p_{j,\infty}(b)\in V_1\cap V_2$, $h_0(x)\in V_1$,
and $f\circ p_{n_j}^{n_l}(x)\in V_2$,
one may find $V\in\mathcal{V}$ with $\{f\circ p_{n_j}^{n_l}(x),h_0(x)\}
\subset V_1\cup V_2\subset V$.

Lastly, we prove (3). Since $|N_{|L|,T_j}|$ is a regular neighborhood
of $|L|$ in $|T_j|$ and hence $|L|$ is a retract of
this regular neighborhood, then there is no loss
of generality in assuming that $f:|N_{|L|,T_j}|\to K$.
Now we apply (2) with $B_j=|N_{|L|,T_j}|$. So there exists $k\geq j$
such that for all $l\geq k$,
there is a map $f^*:P_l\to K$ that extends the composition $f\circ
p_{n_j}^{n_l}|B_l:B_l\to K$ where $$B_l=
(p_{n_j}^{n_l})\inv(B_j)\cap P_l=(p_{n_j}^{n_l})
\inv(|N_{|L|,T_j}|)\cap P_l.$$

Here we treat $p_{n_j}^{n_l}|B_l:B_l\to B_j=|N_{|L|,T_j}|$. By
definition (see(3)), $$E=g\inv(|L|)\subset P_l.$$  Let us
demonstrate that,
\smallskip

$(\dag_1)$  $p_{n_j}^{n_l}(E)\subset|N_{|L|,T_j}|$.
\smallskip

Suppose, for the sake of contradiction, that $x\in E$
and $p_{n_j}^{n_l}(x)\in P_j\setminus|N_{|L|,T_j}|=
|T_j|\setminus|N_{|L|,T_j}|$.
There is a simplex $\sigma\in T_j$ with
$p_{n_j}^{n_l}(x)\in\inter(\sigma)$ and so that $\sigma\cap
|L|=\emptyset$.  Applying the fact that $g$ is a
$T_j$-modification of $p_{n_j}^{n_l}:P_l\to P_j$,
one sees that $g(x)\in\sigma$.  But then $g(x)\notin
|L|$, which is false.  So we get $(\dag_1)$.

The preceding and Lemma \ref{modishotop} show
that the maps $g|E:E\to|L|\subset|N_{|L|,T_j}|$ and
$p_{n_j}^{n_l}|E:E\to|N_{|L|,T_j}|$ are homotopic in $|N_{|L|,T_j}|$.
Hence the compositions $f\circ g|E:E\to K$ and
$f\circ p_{n_j}^{n_l}|E:E\to K$ are homotopic.

Using $(\dag_1)$, one has,
\smallskip

$(\dag_2)$  $E=g\inv(|L|)\cap P_l\subset
(p_{n_j}^{n_l})\inv(|N_{|L|,T_j}|)\cap P_l
=(p_{n_j}^{n_l})\inv(B_j)\cap P_l=B_l$.
\smallskip

The map $f^*:P_l\to K$ extends the composition $f\circ
p_{n_j}^{n_l}:B_l\to K$. Using this and $(\dag_2)$, one sees
that $f^*$ extends the composition $f\circ p_{n_j}^{n_l}|E:E\to K$.
So an application of the homotopy extension theorem gives us the
desired map $g^*$, completing our proof of (3).
\end{proof}

\section[Extension Dimension]
{Extension Dimension}\label{extendim}

In order to strengthen forthcoming results, we
are going to employ the notion of extension dimension.  The reader
can find a good exposition of this in \cite{IR}, but we will
provide all the necessary ideas in this section.

Let $\mathcal{C}$ be a class of spaces and $\mathcal{T}$ a class
of $\CW$-complexes.  For each $K\in\mathcal{T}$, there is
a subclass $[K]_{(\mathcal{C},\mathcal{T})}\subset\mathcal{T}$ which
is called the {\it extension type} of $K$ relative to
$(\mathcal{C},\mathcal{T})$.  Indeed, the extension types form a
decomposition of $\mathcal{T}$.  There is a partial order
$\leq_{(\mathcal{C},\mathcal{T})}$ on the class of extension
types (see page 384 of \cite{IR}). We are not going to explain
it here, but we shall indicate later
how this comes into play for us.

Henceforward, $\mathcal{C}$ will
be the class of metrizable compacta and $\mathcal{T}$ will
be the class of $\CW$-complexes.  We only need to mention that
metrizable spaces are stratifiable ((SP7) on p. 386 of \cite{IR}),
and hence the results of \cite{IR} will apply to our choice of
$\mathcal{C}$.

\begin{proposition}\label{inclusive}For all $X\in\mathcal{C}$
and $K\in\mathcal{T}$, $X\tau K$ if and only if
$X\tau L$ for all $L\in[K]_{(\mathcal{C},\mathcal{T})}$.
\end{proposition}

On the basis of Proposition \ref{inclusive}, if $P$ is an extension type, then one usually writes
$X\tau P$ to mean that $X\tau K$ for all $K\in P$.
Fix $X\in\mathcal{C}$ and consider the class $\mathcal{E}(X)$
consisting of those extension types $P$ relative to
$(\mathcal{C},\mathcal{T})$ with $X\tau P$.  If
$(\mathcal{E}(X),\leq_{(\mathcal{C},\mathcal{T})})$ has an
initial element\footnote{An initial element $s_0\in S$ of a
partially ordered set $(S,\leq)$ is understood in the following
sense: for every $s\in S$, $s_0\leq s$.  Such $s_0$, if it
exists, is unique.}, then that element is called the {\it
extension dimension} of $X$ with respect to
$(\mathcal{C},\mathcal{T})$, written
$\extdim_{(\mathcal{C},\mathcal{T})}X$.
Here are the facts that are salient to us.

By Corollary 5.4 of \cite{IR},

\begin{proposition}\label{existextd}The extension dimension,
$\extdim_{(\mathcal{C},\mathcal{T})}X$, exists for
every metrizable compactum $X$.
\end{proposition}

\begin{proposition}\label{substitute}For all $K$, $L\in\mathcal{T}$,
$[K]_{(\mathcal{C},\mathcal{T})}=[L]_{(\mathcal{C},\mathcal{T})}$
whenever $K\simeq L$.
\end{proposition}

\begin{proposition}\label{polyrep}For every metrizable compactum
$X$, there exists a polyhedron $K$ such that $[|K|]_{(\mathcal{C},\mathcal{T})}=
\extdim_{(\mathcal{C},\mathcal{T})}X$.
\end{proposition}

\begin{lemma}\label{therelation}Let $X$ be a metrizable compactum and $K$ a
$\CW$-complex with $[K]_{(\mathcal{C},\mathcal{T})}=
\extdim_{(\mathcal{C},\mathcal{T})}X$.
Then for every metrizable compactum $Y$, $Y\tau K$ implies that
$Y\tau L$ whenever $L$ is a $\CW$-complex and $X\tau L$.\qed
\end{lemma}

\section[Simplicial Resolution]
{Simplicial Resolution}\label{simpres}

Often when a space is given as the limit of an inverse system,
that inverse system is called a ``resolution'' of the space. It
has been shown by S. Marde\v si\' c (see \cite{Ma}) that in
general it is impossible to resolve a metrizable compactum by
an inverse sequence of triangulated compact polyhedra in which
each bonding map is simplicial with respect to these
triangulations.  But having such a resolution might be valuable
in extension theory.  Starting with a nonempty metrizable
compactum $X$ we are going to prove the existence of a certain
``simplicial'' inverse sequence (see Definition
\ref{inducedinvseq}).  This will be an important step in
reaching our goal of finding a ``replacement'' of the original
space $X$ for the purpose of extension theory. The construction
that yields what we want is found in Lemma
\ref{enhancedbigpicture}.

\begin{definition}\label{homotopyclasses}For a given nonempty
simplicial complex $K$, let $\mathcal{F}(K)$ be a countable set of finite
simplicial complexes having the property that for each $N\in\mathcal{F}(K)$
there exists a subcomplex $N^*$ of $K$ that is simplicially
isomorphic to $N$ and such that for each finite subcomplex
$E$ of $K$, there exists an element $N$ of $\mathcal{F}(K)$ that
is simplicially isomorphic to $E$.
\end{definition}

\begin{lemma}\label{enhancedbigpicture} Let $X\subset I^\infty$ be
compact and nonempty and $K$ be a simplicial complex such that
$[|K|]_{(\mathcal{C},\mathcal{T})}=
\extdim_{(\mathcal{C},\mathcal{T})}X$.
Select a bijective
function $\eta:\N\to\N\times\N$, denote
\smallskip

$(*_1)$  $\eta(j)=(s_j,t_j)$ for each $j\in\N$,
\smallskip

and require that,
\smallskip

$(*_2)$  $s_j\leq j$ for each $j\in\N$.
\smallskip

Then there exists a sequence $(\mathcal{S}_j)$,
$\mathcal{S}_j=(n_j, P_j,\epsilon_j,\widetilde T_j,f_j^{j+1},
\delta_j,T_j,\varphi_j,g_j^{j+1})$, such that $(n_j)$ is an
increasing sequence in $\N$, $(P_j)$ is a sequence
of compact polyhedra, for each $j$, $P_j\subset I^{n_j}$, $p_{n_j}^{n_{j+1}}
(P_{j+1})\subset P_j$, $f_j^{j+1}:P_{j+1}\to P_j$ is a map,
$(\epsilon_j)$ and $(\delta_j)$ are sequences of positive real
numbers, and $\widetilde T_j$ and $T_j$ are triangulations of
$P_j$, $T_j$ being a subdivision of $\widetilde T_j$. The
sequence will be constructed so that $f_j^{j+1}:|\widetilde
T_{j+1}|\to|T_j|$ is a simplicial approximation of
$p_{n_j}^{n_{j+1}}|P_{j+1}:P_{j+1}\to P_j$,
$\varphi_j:|T_j|\to |\widetilde T_j|$ is a
simplicial approximation of the identity map of $P_j=|T_j|$, and we shall define
$g_j^{j+1}=f_j^{j+1}\circ\varphi_{j+1}:|T_{j+1}|\to|T_j|$. For
each $j\in\N$, we shall index the
countable set $\mathfrak{F}_j=\bigcup\{\mathcal{H}(D,|N|)
\,|\,(D,N)\in\mathcal{B}(P_j)\times\mathcal{F}(K)\}$ as
$\{f_{j,k}\,|\,k\in\N\}$ where $f_{j,k}:
D_{j,k}\to|N_{j,k}|$.  All eight of the following conditions will be
satisfied simultaneously by these choices:
\begin{enumerate}
\item if $j>1$, $u$, $v\in I^\infty$ and $\rho(u,v)<\epsilon_j$, then
$\rho(p_{n_s,\infty}(u),p_{n_s,\infty}(v))<\delta_s$ for all $1\leq s<j$,
\item $5\cdot2^{-n_j}<\epsilon_j$,\item $\delta_j<2^{1-n_j}$,
\item $\mesh T_j<\frac{\delta_j}{2}$,\item the map $f_{s_j,t_j}\circ g_{s_j}^j|(g_{s_j}^j)
\inv(D_{s_j,t_j}):(g_{s_j}^j)\inv(D_{s_j,t_j})\to|N_{s_j,t_j}|$
extends to a map $\hat f_{s_j,t_j}:|E_{s_j,t_j}|\to
|N_{s_j,t_j}|$, where $E_{s_j,t_j}$ is the simplicial
neighborhood of $(g_{s_j}^j)\inv(D_{s_j,t_j})$ in $T_j$,
\item  $|E_{s_j,t_j}^*|$ is a regular
neighborhood of $|E_{s_j,t_j}|$ in $T_j$ where $E_{s_j,t_j}^*$
is the simplicial neighborhood of $|E_{s_j,t_j}|$ in $T_j$,
\item for each $x\in X$, there exists $v_{x,j}\in\widetilde
T_j^{(0)}$ such that $\overline
N(p_{n_j,\infty}(x),2\delta_j)\cap
P_j\subset\overline\st(v_{x,j},\widetilde T_j)
\subset\overline N(p_{n_j,\infty}(x),\epsilon_j)\cap P_j$, and
\item whenever $j>1$ and $1\leq s<j$, the simplicial map
$g_s^j:|T_j|\to|T_s|$ is a simplicial approximation of
the restricted projection $p_{n_s}^{n_j}|{P_j}:P_j\to P_s$.
\end{enumerate}
\end{lemma}

\begin{proof}Let $(n_j)$ be an increasing sequence in $\N$ as in
Lemma \ref{intersectinH}.  There is also a sequence $(P_j)$ of
polyhedra given there, and for each of these we select
the set $\mathfrak{F}_j$ indexed
as in the hypothesis. In the ensuing proof we might
have to find an increasing sequence $(j_i)$ in $\N$ and
replace $(n_j)$ with the  nondecreasing subsequence $(m_i)$,
$m_i=n_{j_i}$.  To conserve notation, we shall not introduce
the symbols $m_i$, but rather will simply rename $n_j$ as needed
and rely on Lemma \ref{intersectinH} and Corollary \ref{picture1}
to help fill in any gaps that this might seem to present.  A similar
case will also apply to the polyhedra $P_j$.  Note that $n_1=1$.

Choose $\epsilon_1=5$ and a triangulation $\widetilde T_1$ of
$P_1=I^1$ with $\mesh\widetilde
T_1<\frac{5}{2}=\frac{\epsilon_1}{2}$. Let $\lambda_1$ be a
Lebesgue number of the open cover
$\mathcal{U}_1=\{\st(v,\widetilde T_1)\,|\, v\in\widetilde
T_1^{(0)}\}$ of $P_1=I^1$, and pick $0<
\delta_1=\min\{\frac{\lambda_1}{3},2\inv\}$.

Using $(*_1)$, $(*_2)$, observe that $\eta(1)=(s_1,t_1)=(1,t_1)$
for some $t_1\in\N$. The map
$f_{1,t_1}=f_{s_1,t_1}:D_{s_1,t_1}\to|N_{s_1,t_1}|$ lies in
$\mathfrak{F}_{s_1}=\mathfrak{F}_1$, so its domain $D_{s_1,t_1}$ is a closed
subset of $P_1$ and its range is a compact polyhedron.
Hence $f_{s_1,t_1}$ extends to a map
$f_{s_1,t_1}^*:D_{s_1,t_1}^*\to|N_{s_1,t_1}|$, where
$D_{s_1,t_1}^*$ is a neighborhood  of $D_{s_1,t_1}$ in $P_1$.
Select a triangulation $T_1$ of $P_1$ that refines $\widetilde
T_1$ and so that $\mesh T_1<\frac{\delta_1}{2}$. We may also assume
about the triangulation $T_1$ that $|E_{s_1,t_1}|\subset D_{s_1,t_1}^*$,
where $E_{s_1,t_1}$ is the simplicial neighborhood of $D_{s_1,t_1}$
with respect to $T_1$.  Employing $f_{s_1,t_1}^*$ and the preceding,
we see that $f_{s_1,t_1}$ extends to a map $\hat f_{s_1,t_1}:|E_{s_1,t_1}|
\to|N_{s_j,t_j}|$. We may go even further and require that
$|E_{s_1,t_1}^*|$ is a regular neighborhood of $|E_{s_1,t_1}|$
in $|T_1|$, where $E_{s_1,t_1}^*$ is the simplicial
neighborhood of $|E_{s_1,t_1}|$ in $T_1$.  Let
$\varphi_1:|T_1|\to|\widetilde T_1|$ be a simplicial
approximation of the identity map of $P_1$.

Since $2\delta_1<\lambda_1$, then for each $x\in X$ we may
choose $v_{x,1}\in\widetilde T_1^{(0)}$ such that $\overline
N(p_{n_1,\infty}(x),2\delta_1)\subset\st(v_{x,1},\widetilde
T_1)$.  For a given $y\in \overline\st(v_{x,1},\widetilde
T_1)$, $\rho(y,p_{n_1,\infty}(x))\leq2\mesh\widetilde
T_1<\epsilon_1$, so $\overline\st(v_{x,1},\widetilde
T_1)\subset\overline N(p_{n_1,\infty}(x),\epsilon_1)=\overline
N(p_{n_1,\infty}(x),\epsilon_1)\cap P_1$. It follows that,
$\overline N(p_{n_1,\infty}(x),2\delta_1)=\overline
N(p_{n_1,\infty}(x),2\delta_1)\cap P_1\subset
\overline\st(v_{x,1},\widetilde T_1)\subset\overline
N(p_{n_1,\infty}(x),\epsilon_1)\cap P_1$.

Taking the preceding as the first step in a recursion,
then all of (1)--(8) hold true for $j=1$.
Now assume that $i\in\N$, and that we have completed the
construction of $\mathcal{S}_j$ for each $1\leq j\leq i$
in accordance with (1)--(8).
Aside from the inductive assumptions, and as mentioned above,
we insist on one proviso.  We shall agree that
the numbers $\{n_1,\dots,n_i\}$ are in reality $\{n_{j_1},\dots,n_{j_i}\}$
where $j_1<\dots<j_i$, so we get a finite
nondecreasing subsequence of the given infinite sequence $(n_j)$
(retaining the symbol $n_j$ in order to conserve notation).
Now we proceed for the $(i+1)$-step.

Applying the uniform continuity of
the coordinate projections of $I^\infty$, select $0<\epsilon_{i+1}$ so
that if $u$, $v\in I^\infty$ and $\rho(u,v)<\epsilon_{i+1}$, then for
each $1\leq s\leq i$, $\rho(p_{n_s,\infty}(u),p_{n_s,\infty}(v))<\delta_s$.
This achieves (1) for $j=i+1$.

Choose $n_{i+1}>n_i$ so that (2) is satisfied for $j=i+1$.
As a consequence of Lemma \ref{intersectinH}(2), $p_{n_s}^{n_j}(P_j)\subset P_s$
for all $1\leq s<j\leq i+1$. So we can write the restrictions of
the projections as $p_{n_s}^{n_j}|{P_j}:P_j\to P_s$, for such $s$.
Keep in mind that $P_{i+1}\subset I^{n_{i+1}}$.
Select a triangulation $\widetilde T_{i+1}$ of $P_{i+1}$ with $\mesh
\widetilde T_{i+1}<\frac{\epsilon_{i+1}}{2}$ and so that
at the same time we may find a simplicial approximation $f_i^{i+1}:
|\widetilde T_{i+1}|\to|T_i|$
to the map $p_{n_i}^{n_{i+1}}|{P_{i+1}}:P_{i+1}\to P_i$. Let
$\lambda_{i+1}$ be a Lebesgue number of the open cover
$\mathcal{U}_{i+1}=\{\st(v,\widetilde T_{i+1})\,|\,v
\in\widetilde T_{i+1}^{(0)}\}$ of $P_{i+1}$, and pick
$0<\delta_{i+1}=\min\{\frac{\lambda_{i+1}}{3},
2^{1-n_{i+1}}\}$.  This gives us (3) for $j=i+1$.

Using $(*_1)$, $(*_2)$, observe that $\eta(i+1)=(s_{i+1},t_{i+1})$
for some $t_{i+1}\in\N$ and where $s_{i+1}\leq i+1$. The map
$f_{s_{i+1},t_{i+1}}:D_{s_{i+1},t_{i+1}}\to|N_{s_{i+1},t_{i+1}}|$ lies in
$\mathfrak{F}_{s_{i+1}}$, so its domain $D_{s_{i+1},t_{i+1}}$ is a closed
subset of $P_{s_{i+1}}$ and its range is a compact polyhedron.  So the map
$f_{s_{i+1},t_{i+1}}\circ g_{s_{i+1}}^{i+1}|
(g_{s_{i+1}}^{i+1})\inv(D_{s_{i+1},t_{i+1}}):(g_{s_{i+1}}^{i+1})
\inv(D_{s_{i+1},t_{i+1}})\to|N_{s_{i+1},t_{i+1}}|$ extends to a map
$f_{s_{i+1},t_{i+1}}^*:D_{s_{i+1},t_{i+1}}^*\to
|N_{s_{i+1},t_{+1}}|$, where $D_{s_{i+1},t_{i+1}}^*$ is a
neighborhood of $(g_{s_{i+1}}^{i+1})\inv(D_{s_{i+1},t_{i+1}})$
in $P_{i+1}$.  Select a triangulation $T_{i+1}$ of $P_{i+1}$
that refines $\widetilde T_{i+1}$ and so that $\mesh
T_{i+1}<\frac{\delta_{i+1}}{2}$; this accomplishes (4) for
$j=i+1$.   We may also assume
about the triangulation $T_{i+1}$ that $|E_{s_{i+1},t_{i+1}}|
\subset D_{s_{i+1},t_{i+1}}^*$,
where $E_{s_{i+1},t_{i+1}}$ is the simplicial neighborhood of
$(g_{s_{i+1}}^{i+1})\inv(D_{s_{i+1},t_{i+1}})$
with respect to $T_{i+1}$. Employing $f_{s_{i+1},t_{i+1}}^*$ and the preceding,
we see that $f_{s_{i+1},t_{i+1}}$ extends to a map $\hat f_{s_{i+1},t_{i+1}}:
|E_{s_{i+1},t_{i+1}}|\to|N_{s_{i+1},t_{i+1}}|$, so we get (5).
We may go even further and require that
$|E_{s_{i+1},t_{i+1}}^*|$ is a regular neighborhood of $|E_{s_{i+1},t_{i+1}}|$
in $|T_1|$, where $E_{s_{i+1},t_{i+1}}^*$ is the simplicial
neighborhood of $|E_{s_{i+1},t_{i+1}}|$ in $T_{i+1}$.  This yields (6).
Let $\varphi_{i+1}:|T_{i+1}|\to|\widetilde T_{i+1}|$ be
a simplicial approximation of the identity map of $P_{i+1}$.

Since $2\delta_{i+1}<\lambda_{i+1}$, then for each $x\in X$ we may choose
$v_{x,i+1}\in\widetilde T_{i+1}^{(0)}$ such that $\overline
N(p_{n_{i+1},\infty}(x),2\delta_{i+1})\cap P_{i+1}
\subset\st(v_{x,i+1},\widetilde T_{i+1})$.  For a
given $y\in \overline\st(v_{x,i+1},\widetilde T_{i+1})$,
$\rho(y,p_{n_{i+1},\infty}(x))\leq2\mesh\widetilde T_{i+1}<\epsilon_{i+1}$, so
$\overline\st(v_{x,i+1},\widetilde T_{i+1})\subset\overline
N(p_{n_{i+1},\infty}(x),\epsilon_{i+1})\cap P_{i+1}$.
It follows that, $\overline
N(p_{n_{i+1},\infty}(x),2\delta_{i+1})\cap P_{i+1}\subset
\overline\st(v_{x,i+1},\widetilde T_{i+1})\subset\overline
N(p_{n_{i+1},\infty}(x),\epsilon_{i+1})\cap P_{i+1}$.
We have achieved (7) for $j=i+1$.  Put $g_i^{i+1}=f_i^{i+1}\circ
\varphi_{i+1}:|T_{i+1}|\to|T_i|$.  We see that $g_i^{i+1}$ is
simplicial. It follows that for each $1\leq s<i+1$, $g_s^{i+1}:
|T_{i+1}|\to|T_s|$ is a simplicial map.

The next thing to establish is (8) for $j=i+1$.
We have to show that if $1\leq s<i+1$, then
$g_s^{i+1}:|T_{i+1}|\to|T_s|$ is a simplicial approximation
of $p_{n_s}^{n_{i+1}}|P_{i+1}:P_{i+1}\to P_s$.  The inductive assumption
is that (8) is true whenever $1\leq i_0<i$
and $j=i_0+1$.  This means that if $1\leq s<i_0+1$, then
$g_s^{i_0+1}:|T_{i_0+1}|\to|T_s|$ is a simplicial approximation
of $p_{n_s}^{n_{i_0+1}}|P_{i_0+1}:P_{i_0+1}\to P_s$.
Let us determine a fact that will be useful twice.
\vspace{.05in}

$(\dag)$  Let $x\in P_{i+1}$.  There are a unique $\sigma\in
T_i$ with $p_{n_i}^{n_{i+1}}(x)\in\inter(\sigma)$, $\tau\in
T_{i+1}$ with $x\in\inter(\tau)$, and $\tau^*\in\widetilde
T_{i+1}$ with $x\in\inter(\tau^*)$.  Thus, $\tau\subset\tau^*$,
and $\varphi_{i+1}(x)\in\tau^*$. Since $f_i^{i+1}:|\widetilde
T_{i+1}|\to|T_i|$ is a simplicial approximation of
$p_{n_i}^{n_{i+1}}|P_{i+1}:P_{i+1}\to P_i$, $\sigma\in T_i$,
$x\in\inter(\tau^*)$, and
$p_{n_i}^{n_{i+1}}(x)\in\inter(\sigma)$, then
$f_i^{i+1}(\tau^*)\subset\sigma$.

Consider first the case that $s=i$.  We have to show that the
simplicial map $g_i^{i+1}:|T_{i+1}|\to|T_i|$ is a simplicial
approximation of $p_{n_i}^{n_{i+1}}|P_{i+1}:P_{i+1}\to P_i$.
The result of $(\dag)$ shows that
$g_i^{i+1}(x)=f_i^{i+1}\circ\varphi_{i+1}(x)\in
f_i^{i+1}(\tau^*)\subset\sigma$, as required for (8).

The other case is that $s<i$.  The inductive assumption gives us:
\vspace{.03in}

$(*)$  $g_s^i:|T_i|\to|T_s|$ is a simplicial approximation of
$p_{n_s}^{n_i}|P_i:P_i\to P_s$.
\vspace{.03in}

Choose $\kappa\in T_s$ such that $p_{n_s}^{n_{i+1}}(x)
=p_{n_s}^{n_i}\circ p_{n_i}^{n_{i+1}}(x)\in\inter\kappa$. By
$(*)$, $g_s^i\circ p_{n_i}^{n_{i+1}}(x)\in\kappa$. Take
$\sigma$ from $(\dag)$. Then $p_{n_i}^{n_{i+1}}(x)\in
\inter(\sigma)$, $\sigma\in T_i$, and $g_s^i\circ
p_{n_i}^{n_{i+1}}(x)\in\kappa$, so since $g_s^i$ is simplicial,
one has that $g_s^i(\sigma)\subset\kappa$. In the case $s=i$,
we showed that $g_i^{i+1}(x)\in\sigma$. It follows that
$g_s^{i+1}(x)=g_s^i\circ g_i^{i+1}(x)\in
g_s^i(\sigma)\subset\kappa$, which is what we need to complete
the proof of (8).
\end{proof}

\begin{definition}\label{inducedinvseq}For each nonempty
metrizable compactum $X$, select an embedding $X\hookrightarrow I^\infty$,
let $K$ be a simplicial complex such that
$[|K|]_{(\mathcal{C},\mathcal{T})}=\extdim_{(\mathcal{C},\mathcal{T})}X$,
and choose a sequence $(\mathcal{S}_j)$ as
in $\mathrm{Lemma\,\,\ref{enhancedbigpicture}}$. Then
$\mathbf{Z}=(|T_j|,g_j^{j+1})$ will be called an {\bf induced inverse
sequence} for $X$.  We shall usually use the term induced inverse
sequence without reference to the sequence $(\mathcal{S}_j)$, but
the latter will always be available if needed in an argument.
\end{definition}

\section[Extension-Theoretic Property of an Induced Inverse Sequence]
{Extension-Theoretic Property of an Induced Inverse Sequence}
\label{propssimpres}

Theorem \ref{ZistauK} provides the second step in showing how
to replace a given nonempty metrizable compactum with a better
one for the purposes of extension theory.

\begin{theorem}\label{ZistauK}Let $X$ be a nonempty metrizable
compactum and suppose that $\mathbf{Z}=
(|T_j|,g_j^{j+1})$ is an induced inverse sequence for $X$ as in
$\mathrm{Definition\,\ref{inducedinvseq}}$. Put
$Z=\lim\mathbf{Z}$.  If $K_0$ is a
$\CW$-complex and $X\tau K_0$, then $Z\tau K_0$.
\end{theorem}

\begin{proof}Let us incorporate all the notation from Lemma
\ref{enhancedbigpicture}.  It follows from Theorem \ref{inclusive}, that we only have to show that $Z\tau|K|$. Let $j\in\N$, $D\in\mathcal{B}(P_j)$
($P_j=|T_j|$), and $f:D\to|K|$ a map.
We are going to find $k\geq j$ and a map $g:P_k\to|K|$
such that for all $x\in(g_j^k)\inv(D)$, $g(x)=f\circ g_j^k(x)$.
According to Proposition \ref{baseextension}, that will
complete our proof.

There is a finite subcomplex $F$ of $K$ with $f(D)\subset|F|$.
Choose an element $N$ of $\mathcal{F}(K)$ that is simplicially
isomorphic to $F$ and let $\phi:|F|\to|N|$ and $\psi:|N|\to|F|$
be inverse homeomorphisms.  Put $f_0=f:D\to|F|$, and then $f^*=
\phi\circ f_0:D\to|N|$.  There exists $l\in\N$ so that
$N_{j,l}=N$, $D_{j,l}=D$, and $f_{j,l}\simeq f^*$, where
$f_{j,l}:D_{j,l}\to|N_{j,l}|$. (See Definition
\ref{homotopyclasses} and $\mathfrak{F}_j$ of Lemma
\ref{enhancedbigpicture}).

Since $\eta$ of Lemma \ref{enhancedbigpicture} is surjective,
choose $j^*$ so that $\eta(j^*)=(j,l)$. This means that
$(j,l)=(s_{j^*},t_{j^*})$ and, of course, $s_{j^*}=j\leq j^*$
(see $(*_1)$ and $(*_2)$ of Lemma \ref{enhancedbigpicture}). By
this, one can see that $D=D_{s_{j^*},t_{j^*}}$,
$N_{j,l}=N_{s_{j^*},t_{j^*}}$,
$f_{j,l}=f_{s_{j^*},t_{j^*}}\simeq f^*:D_{s_{j^*},t_{j^*}}\to|N_{s_{j^*},t_{j^*}}|$,
and that $\psi:|N_{s_{j^*},t_{j^*}}|\to|F|$. With $j^*$ in
place of $j$, employ (5) and (6) of Lemma
\ref{enhancedbigpicture}. The map $f_{s_{j^*},t_{j^*}}\circ
g_{s_{j^*}}^{j^*}|(g_{s_{j^*}}^{j^*})
\inv(D_{s_{j^*},t_{j^*}}):
(g_{s_{j^*}}^{j^*})\inv(D_{s_{j^*},t_{j^*}})
\to|N_{s_{j^*},t_{j^*}}|$ extends to a map $\hat
f_{s_{j^*},t_{j^*}}:|E_{s_{j^*},t_{j^*}}|\to
|N_{s_{j^*},t_{j^*}}|$ where $E_{s_{j^*},t_{j^*}}$ is the
simplicial neighborhood of
$(g_{s_{j^*}}^{j^*})\inv(D_{s_{j^*},t_{j^*}})$ in $T_{j^*}$,
and $|E_{s_{j^*},t_{j^*}}^*|$ is a regular neighborhood of
$|E_{s_{j^*},t_{j^*}}|$ in $|T_{j^*}|$, where $
E_{s_{j^*},t_{j^*}}^*$ is the simplicial neighborhood of
$|E_{s_{j^*},t_{j^*}}|$ in $T_{j^*}$.  Now put $\hat f=
\psi\circ\hat f_{s_{j^*},t_{j^*}}:|E_{s_{j^*},t_{j^*}}|\to|F|
\subset|K|_{\CW}$.

Next we put Lemma \ref{endset} into play.  Replace $j$ by $j^*$
and $L$ by $E_{s_{j^*},t_{j^*}}$ in (3) of that Lemma, and make
the observation that the simplicial neighborhood $
E_{s_{j^*},t_{j^*}}^*$ of $|E_{s_{j^*},t_{j^*}}|$ in $T_{j^*}$
produces the desired regular neighborhood $|
E_{s_{j^*},t_{j^*}}^*|$ of $|E_{s_{j^*},t_{j^*}}|$ in $T_{j^*}$. Assume that
$k\geq j^*$ is as in $(2)$ of Lemma \ref{endset} with $B_{j^*}
=|E_{s_{j^*},t_{j^*}}|$ playing the role of $B_j$, and use
$g_{j^*}^k:P_k\to|T_{j^*}|$ as the map $g$ which is a
$T_{j^*}$-modification of $p_{n_{j^*}}^{n_k}|P_k:P_k\to
P_{j^*}=|T_{j^*}|$ (in this instance we just take $l=k$ for the
$l$ of Lemma \ref{endset}). The map $f$ of (3) of the cited
lemma will be $\hat f$. In this case we shall have
$E=(g_{j^*}^k)\inv(|E_{s_{j^*},t_{j^*}}|)\subset P_k$.

From all this, we conclude that there is a map
$g^*:P_k\to|K|$ that extends the composition $\hat f\circ
g_{j^*}^k|E:E\to|K|$.  Since $j=s_{j^*}\leq j^*\leq k$, and
$(g_{s_{j^*}}^{j^*})
\inv(D_{s_{j^*},t_{j^*}})\subset|E_{s_{j^*},t_{j^*}}|$, then
$(g_j^k)\inv(D)=(g_{s_j^*}^k)\inv(D_{s_{j^*},t_{j^*}})\subset
E$. Therefore according to the first paragraph of this proof,
it is simply a matter of showing that $\hat
f\circ g_{j^*}^k|(g_{s_{j^*}}^k)\inv(D_{s_{j^*},t_{j^*}})
:(g_{s_{j^*}}^k)\inv(D_{s_{j^*},t_{j^*}})\to|K|$ is
homotopic to the map $f\circ g_{s_{j^*}}^k|
(g_{s_{j^*}}^k)\inv(D_{s_{j^*},t_{j^*}})
:(g_{s_{j^*}}^k)\inv(D_{s_{j^*},t_{j^*}})\to|K|$.  Since
$g_{s_{j^*}}^k=g_{s_{j^*}}^{j^*}\circ g_{j^*}^k$, then this
comes to showing that $\hat
f|(g_{s_{j^*}}^{j^*})\inv(D_{s_{j^*},t_{j^*}}):
(g_{s_{j^*}}^{j^*})\inv(D_{s_{j^*},t_{j^*}})\to|K|$ is
homotopic to $$f\circ
g_{s_{j^*}}^{j^*}|(g_{s_{j^*}}^{j^*})\inv(D_{s_{j^*},t_{j^*}}):
(g_{s_{j^*}}^{j^*})\inv(D_{s_{j^*},t_{j^*}})\to|K|.$$

Recall that $\hat f=\psi\circ\hat f_{s_{j^*},t_{j^*}}$, and
$\hat
f_{s_{j^*},t_{j^*}}|(g_{s_{j^*}}^{j^*})\inv(D_{s_{j^*},t_{j^*}})=
f_{s_{j^*},t_{j*}}\circ g_{s_{j*}}^{j^*}|(g_{s_{j^*}}^{j^*})
\inv(D_{s_{j^*},t_{j^*}}) $.  But
$f_{s_{j^*},t_{j^*}}=f^*=\phi\circ f_0=\phi\circ
f|D_{s_{j^*},t_{j^*}}$.  Since $\phi$ and $\psi$ are inverse
homeomorphisms, then we get the result by substitution.
\end{proof}

\section[Adjustments of an Induced Inverse Sequence]
{Adjustments of an Induced Inverse Sequence}
\label{adjust}

We want to increase the flexibility of our work so that it can
have multiple applications.  The major development of this
section towards that goal is Lemma \ref{altAJR}. For the
remainder of this section, let $X$ be a nonempty metrizable compactum,
$\mathbf{Z}=(|T_j|,g_j^{j+1})$ an induced inverse
sequence for $X$, and $Z=\lim\mathbf{Z}$. Fix $n\geq0$,
including $n=\infty$, and for each $j\in\N$, let
$M_j\subset|T_j|\subset I^{n_{j}}$ be a nonempty closed subset
such that $g_j^{j+1}(M_{j+1})\subset M_j$.  For each $i$, we
shall write
\smallskip

$(*_1)$  $\hat g_i^k=g_i^k\big||T_k^{(n)}|:
|T_k^{(n)}|\to|T_i^{(n)}|$ whenever $i\leq k$,
\smallskip

$(*_2)$
$\widehat\varphi_i=\varphi_i\big||T_i^{(n)}|:|T_i^{(n)}|
\to|\widetilde T_i^{(n)}|$,
\smallskip

$(*_3)$ $\hat f_i^{i+1}=f_i^{i+1}\big||\widetilde
T_{i+1}^{(n)}|: |\widetilde T_{i+1}^{(n)}|\to|T_i^{(n)}|$, and
\smallskip

$(*_4)$  $\alpha_i$ for the inclusion $|\widetilde
T_i^{(n)}|\hookrightarrow|\widetilde T_i^{(n+1)}|$.
\smallskip

Of course when $i<k$, then $g_i^k$ factors as $g_i^s\circ g_s^k$
whenever $i\leq s\leq k$.  Such a property is ``inherited'' by
$\hat g_i^k$.  Let us make some notes about this and others that
can be gleaned from the preceding.  For each $i<k$,
\smallskip

$(*_5)$  $\hat g_i^k=\hat g_i^{k-1}\circ\hat g_{k-1}^k$,
\smallskip

$(*_6)$  $\hat g_{k-1}^k=
\hat f_{k-1}^k\circ\widehat\varphi_k$, and
\smallskip

$(*_7)$  $\hat g_i^k=\hat g_i^{k-1}\circ
\hat f_{k-1}^k\circ\widehat\varphi_k$.
\smallskip

We need to name some sets.

\begin{definition}\label{somesets} For each
$x\in X$ and $j\in\N$, let $v_{x,j}$ be the vertex of
$\widetilde T_j$ in
$\mathrm{Lemma\,\,\ref{enhancedbigpicture}}(5)$, and define:
\begin{enumerate}\item $B_{x,j}=\overline
N(p_{n_j,\infty}(x),2\delta_{j})\cap M_j$,\item
$B_{x,j}^+=\overline \st(v_{x,j},\widetilde T_j)\cap
M_j$,\item $B_{x,j}^\#=\overline
N(p_{n_j,\infty}(x),\epsilon_j)\cap M_j$, and\item
$D_{x,j}=\overline\st(v_{x,j},\widetilde
T_j)\cap|\widetilde T_j^{(n)}|$.
\end{enumerate}
\end{definition}

Notice that $\widetilde T_j$ induces a unique triangulation
$\widetilde L_{x,j}$ of $\overline\st(v_{x,j},\widetilde T_j)$.
So one has, from  Definition \ref{somesets}(4), that $D_{x,j}=
|\widetilde L_{x,j}^{(n)}|$.  Since
$\overline\st(v_{x,j},\widetilde T_j)$ is contractible and
$\dim\widetilde L_{x,j}^{(n)}\leq n$, then $D_{x,j}$ contracts
to a point in $|\widetilde L_{x,j}^{(n+1)}|\subset|\widetilde
L_{x,j}|$. On the other hand, since $T_j$ is a subdivision of
$\widetilde T_j$, it follows that $T_j$ induces a unique
triangulation $L_{x,j}$ of $\overline\st(v_{x,j},\widetilde
T_j)$, and $|\widetilde L_{x,j}^{(n)}|\subset|L_{x,j}^{(n)}|$.
So if $M_j=|T_j^{(n)}|$, then $B_{x,j}^+=|L_{x,j}^{(n)}|$, and
$D_{x,j}\subset B_{x,j}^+$. Since
$\varphi_j:|T_j|\to|\widetilde T_j|$ in Lemma
\ref{enhancedbigpicture} is a simplicial approximation to the
identity on $|T_j|=|\widetilde T_j|$, we get that
$\widehat\varphi_j (|T_j^{(n)}|)=|\widetilde T_j^{(n)}|$ and
$\widehat\varphi_j(B_{x,j}^+)= D_{x,j}$.   Another fact of
importance to us is that since $\overline\st(v_{x,j},\widetilde
T_j)$ is connected, then both $B_{x,j}^+$ and $D_{x,j}$ are
nonempty metrizable continua. Let us record the preceding
information now.

\begin{lemma}\label{DtoB+}Let $B_{x,j}^+$ and $D_{x,j}$ be as in
$\mathrm{Definition\,\ref{somesets}(2,4)}$ where
$M_j=|T_j^{(n)}|$.  Then \begin{enumerate}\item
$B_{x,j}^+=|L_{x,j}^{(n)}|$ and $D_{x,j}=|\widetilde
L_{x,j}^{(n)}|$, $L_{x,j}$ being a subcomplex of $T_j$ and
$\widetilde L_{x,j}$ being a subcomplex of $\widetilde
T_j$,\item $B_{x,j}^+$ and $D_{x,j}$ are nonempty metrizable
continua,\item $D_{x,j}\subset B_{x,j}^+$,\item
$\widehat\varphi_j (|T_j^{(n)}|)=|\widetilde T_j^{(n)}|$,\item
$\widehat\varphi_j(B_{x,j}^+)=D_{x,j}$, and\item$D_{x,j}$
contracts to a point in $|\widetilde L_{x,j}^{(n+1)}|$.\qed
\end{enumerate}
\end{lemma}

Here is the definition that will provide the flexibility
that we mentioned above.

\begin{definition}\label{adjustedinvseq}
Suppose that $(j_i)$ is an increasing sequence in $\N$ and
for each $i$, we are given a map $h_i^{i+1}:M_{j_{i+1}}\to M_{j_i}$
such that $h_i^{i+1}$ is a $T_{j_i}$-modification of the restriction
$g_{j_i}^{j_{i+1}}|M_{j_{i+1}}:M_{j_{i+1}}\to M_{j_i}$.
Then we shall refer to
$\mathbf{M}=(M_{j_i},h_i^{i+1})$ as an {\bf adjustment} of
$\mathbf{Z}$.
\end{definition}

As a consequence of Definition \ref{adjustedinvseq}, $(*_1)$,
Lemma \ref{enhancedbigpicture}(8), and Lemma \ref{modishotop},
we obtain a statement about adjustments.

\begin{lemma}\label{skeletaladj} Let $(j_i)$ be an increasing sequence in $\N$.
Then,
\begin{enumerate}\item
$\mathbf{M}=(|T_{j_i}^{(n)}|,\hat g_{j_i}^{j_{i+1}})$ is an
adjustment of $\mathbf{Z}$, and \item if $(|T_{j_i}^{(n)}|,h_i^{i+1})$ is
an adjustment of $\mathbf{Z}$, then for all $i$, $h_i^{i+1}\simeq
\hat g_{j_i}^{j_{i+1}}$.
\end{enumerate}
\end{lemma}

\begin{lemma}\label{gettrivadj}
If we take $n=\infty$ and $j_i=i$ for all $i$, then the
adjustment $\mathbf{M}=(|T_{j_i}^{(n)}|,\hat
g_i^{i+1})=(|T_i|,g_i^{i+1})$ of $\mathrm{Lemma\,
\ref{skeletaladj}(1)}$ equals $\mathbf{Z}$.\qed
\end{lemma}

\begin{definition}\label{trivadj}We shall call the adjustment
of $\mathbf{Z}$ coming from $\mathrm{Lemma\,\ref{gettrivadj}}$
the {\bf trivial adjustment} of $\mathbf{Z}$.
\end{definition}

This way we can create a theory of adjustments that takes into account
$\mathbf{Z}$ and all other adjustments.
The fundamentals for this are in Lemma \ref{altAJR}.

\begin{definition}\label{adjustseq}Let $(j_i)$ be an increasing
sequence in $\N$,
$\mathbf{M}=(M_{j_i},h_i^{i+1})$ be an adjustment of $\mathbf{Z}$,
and $M=\lim\mathbf{M}$. Whenever a point $w=(a_1,a_2,\dots)\in
M$, then $(a_i)$ is a sequence in $I^\infty$ which we shall call the
sequence {\bf associated with} $w$. For each $i\in\N$, we get a
function $\pi_i:M\to I^\infty$ by setting $\pi_i(w)=a_i$.
\end{definition}

We shall now give a form of Lemma 3.1 of \cite{AJR} that is
suited to our present situation.  In Lemma \ref{altAJR}, we
shall use the notation from Definition \ref{somesets}. It is
worth mentioning that this lemma is independent of the
choice of the fixed $n\geq0$, even $n=\infty$.

\begin{lemma}\label{altAJR}Let $(j_i)$ be an increasing
sequence in $\N$,
$\mathbf{M}=(M_{j_i},h_i^{i+1})$ be an adjustment of $\mathbf{Z}$, and
$M=\lim\mathbf{M}$. Then,
\begin{enumerate}\item for each $w=(a_1,a_2,\dots)\in M$, the
sequence $(a_i)$ in $I^\infty$ associated with $w$ is a
Cauchy sequence in $I^\infty$ whose limit lies in $X$,\item
the sequence $(\pi_i)$, $\pi_i:M\to I^\infty$, is a Cauchy
sequence of maps whose limit $\pi:M\to I^\infty$ is a map
having the property that $\pi(M)\subset X$,\item for each
$x\in X$ and $i\in\N$, $B_{x,j_i}\subset B_{x,j_i}^+\subset
B_{x,j_i}^\#$, and $h_i^{i+1}(B_{x,j_{i+1}}^\#)\subset
B_{x,j_i}$,\item if for each $x\in X$, we define
$\mathbf{M}_x=(B_{x,j_i},h_i^{i+1}|B_{x,j_{i+1}})$,
$\mathbf{M}_x^+=(B_{x,j_i}^+,h_i^{i+1}|B_{x,j_{i+1}}^+)$,
and
$\mathbf{M}_x^\#=(B_{x,j_i}^\#,h_i^{i+1}|B_{x,j_{i+1}}^\#)$,
then each of $\mathbf{M}_x$, $\mathbf{M}_x^+$, and
$\mathbf{M}_x^\#$ is an inverse sequence of metrizable
compacta,\item for all $x\in X$,
$\lim\mathbf{M}_x=\lim\mathbf{M}_x^+=\lim\mathbf{M}_x^\#$,\item
for all $x\in X$, $\pi^{-1}(x)=\lim\mathbf{M}_x$, and\item
if for all $i\in\N$, $|T_{j_i}^{(0)}|\subset M_{j_i}$, then
$\pi:M\to X$ is surjective.
\end{enumerate}
\end{lemma}

\begin{proof}For each $i\in\N$, put $m_i=n_{j_i}$.
Our choice of metric for $I^\infty$ shows that,

\smallskip
$(\dag_1)$  for each $i\in\N$ and $x\in I^\infty$,
$\rho(p_{m_i,\infty}(x),x)\leq2^{-m_i}$.\smallskip

Let $u\in M_{j_{i+1}}\subset
|T|_{j_{i+1}}\subset I^{m_{i+1}}$.  Applying Lemma
\ref{enhancedbigpicture}(8), there is a simplex $\sigma$ of
$T_{j_i}$, and a face $\tau$ of $\sigma$ such that
$p_{m_i}^{m_{i+1}}(u)\in\inter\sigma$ and
$g_{j_i}^{j_{i+1}}(u)\in\inter\tau$.  By Definition
\ref{adjustedinvseq}, $h_i^{i+1}(u)\in\tau$.  Hence
$\{h_i^{i+1}(u),p_{m_i}^{m_{i+1}}(u)\}\subset\sigma$. It
follows from this and Lemma \ref{enhancedbigpicture}(3,4)
that,\smallskip

$(\dag_2)$  for each $i\in\N$ and $u\in M_{j_{i+1}}$,
$\rho(h_i^{i+1}(u),p_{m_i}^{m_{i+1}}(u))<
\frac{\delta_{j_i}}{2}<2^{-m_i}$.\smallskip

Now let $w=(a_1,a_2,\dots)\in M$. Thus,
$\rho(a_i,a_{i+1})=\rho(h_i^{i+1}(a_{i+1}),a_{i+1})$. We know
that $a_{i+1}\in M_{j_{i+1}}$.  So an application of $(\dag_2)$
gives us,\smallskip

$(\dag_3)$  for all $w=(a_1,a_2,\dots)\in M$,
$\rho(h_i^{i+1}(a_{i+1}),p_{m_i}^{m_{i+1}}
(a_{i+1}))<2^{-m_i}$.\smallskip

If one applies the triangle inequality and uses $(\dag_3)$
and $(\dag_1)$ with $x=a_{i+1}$, one gets that
$\rho(a_i,a_{i+1})=\rho(h_i^{i+1}(a_{i+1}),a_{i+1})\leq
\rho(h_i^{i+1}(a_{i+1}),p_{m_i}^{m_{i+1}}(a_{i+1}))
+\rho(p_{m_i}^{m_{i+1}}(a_{i+1}),a_{i+1})
<2^{-m_i}+2^{-m_i}=2^{1-m_i}$ independently of the choice of
$w\in M$.  We record this fact:\smallskip

$(\dag_4)$  Whenever $w=(a_1,a_2,\dots)\in M$ and $i\in\N$, one
has that $\rho(a_i,a_{i+1})<2^{1-m_i}$.\smallskip

Thus $(a_i)$ is a Cauchy sequence in $I^\infty$, and
$(\pi_i)$ is a Cauchy sequence of maps of $M$ to $I^\infty$ whose
limit $\pi$ is a map of $M$ to $I^\infty$.  But for each $i$, $a_i\in
M_{j_i}\subset|T_{j_i}|\subset|T_{j_i}|\times I^\infty_{m_i}$, so an
application of Corollary \ref{picture1} yields that
$\pi(M)\subset X$.  We have established (1) and (2).

Let $x\in X$.  The first part of (3) comes from Lemma
\ref{enhancedbigpicture}(7).  Let $u\in B_{x,j_{i+1}}^\#$. Then,
it is true that $\rho(h_i^{i+1}(u),p_{m_i,\infty}(x))\leq
\rho(h_i^{i+1}(u),p_{m_i}^{m_{i+1}}(u))
+\rho(p_{m_i}^{m_{i+1}}(u),p_{m_i,\infty}(x))$ $=
\rho(h_i^{i+1}(u),p_{m_i}^{m_{i+1}}(u))
+\rho(p_{m_i}^{m_{i+1}}(u),p_{m_i}^{m_{i+1}}\circ
p_{m_{i+1},\infty}(x))$.  In $(\dag_2)$ we recorded that
$\rho(h_i^{i+1}(u),p_{m_i}^{m_{i+1}}(u))<\frac{\delta_{j_i}}{2}$.
But $u\in B_{x,j_{i+1}}^\#$ implies that
$\rho(u,p_{m_{i+1},\infty}(x))<\epsilon_{j_{i+1}}$.  It follows
from Lemma \ref{enhancedbigpicture}(1) that
$\rho(p_{m_i}^{m_{i+1}}(u),p_{m_i}^{m_{i+1}}\circ
p_{m_{i+1},\infty}(x))<\delta_{j_i}$. We therefore conclude that
$h_i^{i+1}(u)\in\overline N(p_{m_i,\infty}(x),2\delta_{j_i})$, so the
second part of (3) is substantiated.  From (3), both (4) and
(5) follow handily.  We must prove (6).

Suppose that $(a_1,a_2,\dots)\in\lim\mathbf{M}_x$ and $i\in\N$.
Then $a_i\in B_{x,j_i}$, so
$\rho(a_i,p_{m_i,\infty}(x))\leq2\delta_{j_i}$. If we apply
this, $(\dag_1)$, and Lemma \ref{enhancedbigpicture}(3), we conclude that
$\rho(a_i,x)\leq\rho(a_i,p_{m_i,\infty}(x))+\rho(p_{m_i,\infty}(x),x)\leq
2\delta_{j_i}+2^{-m_i}<2^{2-m_i}+2^{-m_i}$.  Therefore,
$\pi((a_i))=\lim(a_i)=x$, so we have shown that
$\lim\mathbf{M}_x\subset\pi^{-1}(x)$.  We have to establish the
opposite inclusion.

Suppose that a thread $(a_1,a_2,\dots)$ of $\mathbf{M}$ lies in
$\pi^{-1}(x)$.  For the next, make use of $(\dag_1)$,
$(\dag_4)$, and Lemma \ref{enhancedbigpicture}(2).  For all $i\in\N$,
$\rho(a_i,p_{m_i,\infty}(x))\leq\rho(a_i,x)+\rho(x,p_{m_i,\infty}(x))\leq
\sum_{k=i}^\infty\rho(a_k,a_{k+1})+2^{-m_i}\leq
\sum_{k=i}^\infty2^{1-m_k}+2^{-m_i}\leq2\cdot2^{1-m_i}+2^{-m_i}=
5\cdot2^{-m_i}<\epsilon_{m_i}$.  This puts $a_i\in B_{x,i}^\#$.
So, $(a_1,a_2,\dots)\in\lim\mathbf{M}_x^\#=\lim\mathbf{M}_x$,
as required to complete the proof of (6).

To prove (7), we only need to show that for each $x\in X$,
$\pi^{-1}(x)\neq\emptyset$, and for this we use (6). It is
sufficient to demonstrate that for each $i\in\N$,
$B_{x,j_i}\neq\emptyset$.  Lemma \ref{enhancedbigpicture}(4)
states that $\mesh T_{j_i}<\frac{\delta_{j_i}}{2}$.  So
$\overline N(p_{m_i,\infty}(x),2\delta_{j_i})$ has to contain
some $v\in T_{j_i}^{(0)}$.  By hypothesis, $v\in M_{j_i}$.  So
$B_{x,j_i}\neq\emptyset$.
\end{proof}

More technical facts need to be established.

\begin{lemma}\label{technicals}Let $(j_i)$ be an increasing
sequence in $\N$ and $\mathbf{M}=(|T_{j_i}^{(n)}|,
\hat g_{j_i}^{j_{i+1}})$ be the adjustment of $\mathbf{Z}$ as
indicated in $\mathrm{Lemma\,\ref{skeletaladj}(1)}$.
Then for each $i<k$ in $\N$ and $x\in X$,
\begin{enumerate}\item$\hat g_{j_i}^{j_k}=\hat g_{j_i}^{j_{k-1}}\circ
\hat g_{j_{k-1}}^{j_k}$,\item$\hat g_{j_k-1}^{j_k}= \hat
f_{j_k-1}^{j_k}\circ\widehat\varphi_{j_k}$,\item $\hat
g_{j_i}^{j_k}=\hat g_{j_i}^{j_k-1}\circ \hat
f_{j_k-1}^{j_k}\circ\widehat\varphi_{j_k}$,\item$\hat
g_{j_i}^{j_k}(B_{x,j_k}^+)\subset B_{x,j_i}^+$,
and\item$\hat f_{j_k-1}^{j_k}(D_{x,j_k})\subset
B_{x,j_k-1}^+$.
\end{enumerate}
\end{lemma}

\begin{proof}We get (1)-(3) from $(*_5)$-$(*_7)$ and (4)
from Lemma \ref{altAJR}(4) when applying the adjustment
$\mathbf{M}$ to the inverse sequence $\mathbf{M}_x^+$. One
arrives at (5) from Lemma \ref{DtoB+}(5), (2), and (4) as applied
to $\hat g_{j_k-1}^{j_k}$.
\end{proof}

\begin{lemma}\label{homofadjust}Let
$(|T_{j_i}^{(n)}|,h_i^{i+1})$ be an adjustment of $\mathbf{Z}$,
$x\in X$, and $i\in\N$.  Then,\begin{enumerate}\item
$h_i^{i+1}(B_{x,j_{i+1}}^+)\subset B_{x,j_i}^+$, and \item
$h_i^{i+1}|B_{x,j_{i+1}}^+\simeq\hat
g_{j_i}^{j_{i+1}}|B_{x,j_{i+1}}^+$ as maps to $B_{x,j_i}^+$.
\end{enumerate}
\end{lemma}

\begin{proof}Lemma \ref{altAJR}(4) implies (1).
By Lemma \ref{DtoB+}(1),
$B_{x,j_i}^+=|L_{x,j_i}^{(n)}|$ where $L_{x,j_i}$ is a
subcomplex of $T_{j_i}$. By this, Lemma \ref{technicals}(4),
the fact that $h_i^{i+1}$ is a $T_{j_i}$-modification of $\hat
g_{j_i}^{j_{i+1}}$, and Lemma \ref{modishotop} we get (2).
\end{proof}

\begin{lemma}\label{anullhom}Let $(j_i)$ be an increasing
sequence in $\N$ and $\mathbf{M}=(|T_{j_i}^{(n)}|,\hat
g_{j_i}^{j_{i+1}})$ be the adjustment of $\mathbf{Z}$ as
indicated in $\mathrm{Lemma\,\ref{skeletaladj}(1)}$. Then for
each $i<k$ in $\N$ and $x\in X$,
\begin{enumerate}\item$\alpha_{j_k}(D_{x,j_k})
\subset|\widetilde
L_{x,j_k}^{(n+1)}|$,\item$\varphi_{j_k}(|L_{x,j_k}^{(n+1)}|)=
|\widetilde L_{x,j_k}^{(n+1)}|$,\item$g_{j_i}^{j_k}
(B_{x,j_k}^+)\subset|L_{x,j_i}^{(n+1)}|$,\item$g_{j_i}^{j_k}
|B_{x,j_k}^+:B_{x,j_k}^+\to|L_{x,j_i}^{(n+1)}|$ is homotopic
to a constant map, and\item in particular if $n=\infty$,
then $g_{j_i}^{j_k} |B_{x,j_k}^+:B_{x,j_k}^+\to B_{x,j_i}^+$
is homotopic to a constant map.
\end{enumerate}
\end{lemma}

\begin{proof}We obtain (1) from the triangulation of
$D_{x,j_k}$ from Lemma \ref{DtoB+}(1) and
$(*_4)$.  Item (2) comes from Lemma  \ref{DtoB+}(5) with $n$
replaced by $n+1$. To get (3), use Lemma
\ref{technicals}(4), Lemma \ref{DtoB+}(1) and the fact that
$|L_{x,j_i}^{(n)}|\subset|L_{x,j_i}^{(n+1)}|$.  To obtain (4),
let us write the map $g_{j_i}^{j_k}$ as a composition.  Let
$t\in B_{x,j_k}^+$.  Then $g_{j_i}^{j_k}(t)=g_{j_i}^{j_k-1}\circ
f_{j_k-1}^{j_k}\circ\alpha_{j_k}\circ\widehat\varphi_{j_k}(t)$
because of Lemma \ref{technicals}(3) and $(*_4)$.

But Lemma \ref{DtoB+}(5) shows that $\widehat\varphi_{j_k}(B_{x,j_k}^+)
=D_{x,j_k}=|\widetilde L_{x,j_k}^{(n)}|$ which contracts to
a point in $|\widetilde L_{x,j_k}^{(n+1)}|$
by (6) of that lemma; so $\alpha_{j_k}:D_{x,j_k}
\to|\widetilde L_{x,j_k}^{(n+1)}|$ is homotopic to
a constant map. Now $f_{j_k-1}^{j_k}(|\widetilde L_{x,j_k}^{(n+1)}|)
\subset|L_{x,j_k-1}^{(n+1)}|$ as a result of applying
Lemma \ref{technicals}(5) in the case that $n$ is replaced
by $n+1$.  Next apply Lemma \ref{technicals}(3) to
complete the argument for (4).  The statement (5) follows
when we choose $n=\infty$ in (4).
\end{proof}

\begin{lemma}\label{lowerstuff}Let
$\mathbf{M}=(|T_{j_i}^{(n)}|,h_i^{i+1})$ be an adjustment of
$\mathbf{Z}$, $M=\lim\mathbf{M}$, and $\pi:M\to X$ the map of
$\mathrm{Lemma\,\,\ref{altAJR}}(2)$. Then for each $x\in X$,
$\pi^{-1}(x)\neq\emptyset$, i.e., $\pi$ is surjective.
\end{lemma}

\begin{proof}Since for each $i$,
$|T_{j_i}^{(0)}|\subset|T_{j_i}^{(n)}|$, then Lemma
\ref{altAJR}(7) yields that $\pi^{-1}(x)\neq\emptyset$.
\end{proof}

\section[Cell-like Map]
{Cell-like Map}\label{celllike}

Theorem \ref{bigmapistrivsh} is the third and last step in
showing how we can obtain
from a given metrizable compactum $X$, a metrizable compactum
$Z$ that is a substitute for $X$ in the sense of extension
theory, where the replacement $Z$ is rich in the (useful)
properties seen in Lemma \ref{enhancedbigpicture}.

\begin{definition}\label{deftrivsh}A compact metrizable space
is said to have {\bf trivial shape} $($or is called {\bf
cell-like}$)$ if it has the shape of a point \cite{MS}. Such a
space always has to be a nonempty continuum.
\end{definition}

\begin{lemma}\label{trivshape}Let $\mathbf{X}=(X_i,h_i^{i+1})$
be an inverse sequence of metrizable compacta and
$X=\lim\mathbf{X}$.  Then $X$ has
trivial shape if for each $i\in\N$, there exists $j>i$
such that $h_i^j:X_j\to X_i$ is null homotopic.\qed
\end{lemma}

\begin{definition}\label{trivshamap}A proper map of one space
to another whose fibers are cell-like
will be called a {\bf cell-like} map.
\end{definition}

One should note that cell-like maps have to be surjective.
Since every finite-dimensional metrizable compactum embeds
in some $\R^n$, then one can use Corollary 5A, p. 145 of \cite{Da},
to justify the next fact.\footnote{This particular property of cell-like
continua plays no role herein.}

\begin{lemma}\label{triveqcell}If $X$ is a finite-dimensional
cell-like space, then there exists $n\in\N$ so that
$X$ can be embedded in $\R^n$ as a cellular subset.\qed
\end{lemma}

\begin{theorem}\label{bigmapistrivsh}Let $X$ be a nonempty
metrizable compactum, $\mathbf{Z}=(|T_i|,g_i^{i+1})$
an induced inverse sequence for
$X$, and $Z=\lim\mathbf{Z}$.  Then the
map $\pi:Z\to X$ of\, $\mathrm{Lemma\,\,\ref{altAJR}}(2)$
under the trivial adjustment is a
cell-like map from the metrizable compactum $Z=\lim\mathbf{Z}$.
Moreover,

$(*)$  for each $\CW$-complex $K_0$ with $X\tau K_0$, $Z\tau
K_0$.
\end{theorem}

\begin{proof}By Definition \ref{trivadj},
$\mathbf{Z}=(|T_i|,g_i^{i+1})$ is the trivial
adjustment of $\mathbf{Z}$. By Lemma \ref{lowerstuff}, $\pi$
is surjective. Fix $x\in X$. From Lemma \ref{altAJR}(5,6),
$\pi\inv(x)=\lim\mathbf{M}_x^+$. In this case, $\mathbf{M}_x^+=
(B_{x,i}^+,g_i^{i+1}|B_{x,i+1}^+)$. By Lemma \ref{anullhom}(5),
every bonding map $g_i^{i+1}|B_{x,i+1}^+:B_{x,i+1}^+\to
B_{x,i}^+$ in $\mathbf{M}_x^+$ is homotopic to a constant map.
So Lemma \ref{trivshape} may be used to complete the proof that
$\pi$ is cell-like. Apply Theorem \ref{ZistauK} to obtain $(*)$.
\end{proof}

\end{document}